\documentclass{amsart}

\usepackage{amssymb}
\usepackage[hidelinks]{hyperref}
\usepackage[textsize=footnotesize]{todonotes}
\usepackage{tikz}
\usetikzlibrary{cd}
\usepackage{float}

\theoremstyle{plain}
\newtheorem{theorem}{Theorem}[section]
\newtheorem{lemma}[theorem]{Lemma}
\newtheorem{corollary}[theorem]{Corollary}
\newtheorem{proposition}[theorem]{Proposition}
\newtheorem{conjecture}[theorem]{Conjecture}

\theoremstyle{definition}
\newtheorem{definition}[theorem]{Definition}
\newtheorem{example}[theorem]{Example}

\theoremstyle{remark}
\newtheorem*{remark}{Remark}

\newcommand{\I}{\mathcal{I}}
\newcommand{\J}{\mathcal{J}}
\newcommand{\fin}{\mathrm{Fin}} 
\newcommand{\Fin}{\mathrm{Fin}}
\DeclareMathOperator{\rk}{rk}

\begin{document}


\title{Inductive limits of ideals}


\author[Adam Kwela]{Adam Kwela}
\address[Adam Kwela]{Institute of Mathematics\\ Faculty of Mathematics\\ Physics and Informatics\\ University of Gda\'{n}sk\\ ul.~Wita  Stwosza 57\\ 80-308 Gda\'{n}sk\\ Poland}
\email{Adam.Kwela@ug.edu.pl}
\urladdr{http://kwela.strony.ug.edu.pl/}


\date{\today}






\begin{abstract}
G. Debs and J. Saint Raymond in 2009 defined the Borel separation rank of an analytic ideal $\mathcal{I}$ ($\text{rk}(\mathcal{I})$) as minimal ordinal $\alpha<\omega_{1}$ such that there is $\mathcal{S}\in\bf{\Sigma^0_{1+\alpha}}$ with $\mathcal{I}\subseteq \mathcal{S}$ and $\mathcal{I}^\star\cap \mathcal{S}=\emptyset$, where $\mathcal{I}^\star$ is the filter dual to the ideal $\mathcal{I}$ (actually, the authors use the dual notion of filters instead of ideals). Moreover, they introduced ideals $\text{Fin}_\alpha$, for all $\alpha<\omega_1$, and conjectured that $\text{rk}(\mathcal{I})\geq\alpha$ if and only if $\mathcal{I}$ contains an isomorphic copy of $\text{Fin}_\alpha$ ($\text{Fin}_\alpha\sqsubseteq\mathcal{I}$). To define $\text{Fin}_\alpha$ in the case of limit ordinals $0<\alpha<\omega_1$, G. Debs and J. Saint Raymond introduced inductive limits of ideals.

We show that the above conjecture is false in the case of $\alpha=\omega$ by constructing an ideal $\text{Fin}'_\omega$ of rank $\omega$ such that $\text{Fin}_\omega\not\sqsubseteq\text{Fin}'_\omega$. However, we show that $\text{Fin}'_\omega\sqsubseteq\mathcal{I}$ is equivalent to $\forall_{n\in\omega}\text{Fin}_n\sqsubseteq\mathcal{I}$. We discuss (indicated by the above result) possible modification of the original conjecture for limit ordinals.
\end{abstract}


\maketitle


\section{Introduction}

Following G. Debs and J. Saint Raymond \cite{Debs}, for an analytic ideal $\I$ we define its Borel separation rank by:
$${\rm rk}(\mathcal{I})=\min\left\{\alpha<\omega_{1}: \textrm{ there is }\mathcal{S}\in\bf{\Sigma^0_{1+\alpha}} \textrm{ such that }\mathcal{I}\subseteq \mathcal{S}\textrm{ and }\mathcal{I}^\star\cap \mathcal{S}=\emptyset\right\},$$
where $\mathcal{I}^\star$ is the filter dual to the ideal $\mathcal{I}$ (actually, authors of \cite{Debs} use the dual notion of filters instead of ideals). In this paper by rank of $\I$ we mean ${\rm rk}(\mathcal{I})$.

In \cite{Debs} and \cite{Reclaw} it is shown that ranks of analytic ideals are important for studying ideal pointwise limits of sequences of continuous functions: ${\rm rk}(\mathcal{I})=\alpha$ if and only if the family of all $\I$-pointwise limits of continuous real-valued functions defined on a given zero-dimensional Polish space is equal to the family of all functions of Borel class $\alpha$ (the definition of $\I$-pointwise limit can be found in \cite{Debs} or \cite{Reclaw}). For more on ranks of ideals see \cite{Kat}, \cite{Debs2} or \cite{zReclawem}.

G. Debs and J. Saint Raymond defined ideals $\text{Fin}_\alpha$, for $0<\alpha<\omega_1$ (for definition see Subsection \ref{subsec}). To define $\text{Fin}_\alpha$ in the case of limit ordinals $0<\alpha<\omega_1$, they introduced inductive limits of ideals. The main motivation of this paper is the conjecture of G. Debs and J. Saint Raymond from 2009:

\begin{conjecture}[{\cite[Conjecture 7.8]{Debs}}]
\label{hip}
Let $\I$ be an analytic ideal. Then $\text{rk}(\mathcal{I})\geq\alpha$ if and only if $\mathcal{I}$ contains an isomorphic copy of $\text{Fin}_\alpha$.
\end{conjecture} 

It is known that this is true for $\alpha=1,2$ (\cite[Theorem 7.5]{Debs} and \cite[Theorem 4]{Reclaw}). Moreover, the implication $\Leftarrow$ is true in general. 

Let us point out a connection of the above conjecture to descriptive complexity of ideals. Each ideal $\Fin_n$, for $n\in\omega$, is $\bf{\Sigma^0_{2n}}$. By a result of G. Debs and J. Saint Raymond from \cite{Debs2}, any ideal $\I$ containing an isomorphic copy of $\Fin_n$ cannot be $\bf{\Pi^0_{2n}}$. Thus, Conjecture \ref{hip} would lead to an interesting result, by giving lower estimate of Borel complexity of ideals of a given finite rank. 

What is more, Conjecture \ref{hip} would imply that we can skip the assumption about zero-dimensionality in the result mentioned in the second paragraph of this Introduction and obtain the following result: ${\rm rk}(\mathcal{I})=\alpha$ if and only if the family of all $\I$-pointwise limits of continuous real-valued functions defined on a given Polish space is equal to the family of all functions of Borel class $\alpha$ (cf. \cite[Corollary 7.6]{Debs}).

In this paper we show that Conjecture \ref{hip} is false in the case of $\alpha=\omega$ by constructing an ideal $\text{Fin}'_\omega$ of rank $\omega$ which does not contain an isomorphic copy of $\text{Fin}_\omega$. However, we discuss a way of modifying the original conjecture for limit ordinals.

The paper is organized as follows. In Section 2 we collect some definitions needed in our further considerations. Section 3 contains a short study of the problem when an inductive limit of ideals is an ideal -- we show that $\fin_\alpha$, for all limit ordinals $0<\alpha<\omega_1$, are ideals despite the fact that \cite[Proposition-Definition 5.4]{Debs}, which defines them, is false. Our main result showing that Conjecture \ref{hip} is false for $\alpha=\omega$ is proved in Section 4. Finally, in Section 5 we show that an ideal $\I$ contains an isomorphic copy of $\fin'_\omega$ if and only if $\I$ contains isomorphic copies of all $\fin_n$, for $n\in\omega$. This result enables us to discuss a possible modification of Conjecture \ref{hip}.

\section{Preliminaries}

\subsection{Basic notions}

A collection $\mathcal{I}$ of subsets of a set $X$ is called an ideal on $X$ if it is closed under subsets and finite unions of its elements. In this case we denote the set $X$ by $\text{dom}(\I)$. All ideals considered in this paper are defined on infinite countable sets. We assume additionally that $\mathcal{P}(X)$ is not an ideal, and that every ideal contains all finite subsets of $X$. By $\fin$ we denote the ideal of all finite subsets of $\omega$.

We treat the power set $\mathcal{P}(X)$ as the space $2^X$ of all functions $f:X\rightarrow 2$ (equipped with the product topology, where each space $2=\left\{0,1\right\}$ carries the discrete topology) by identifying subsets of $X$ with their characteristic functions. Thus, we can talk about descriptive complexity of subsets of $\mathcal{P}(X)$ (in particular, ideals on $X$). 

For $\mathcal{A}\subseteq\mathcal{P}(X)$ we write
$$\mathcal{A}^\star=\{X\setminus A:\ A\in\mathcal{A}\}.$$
If $\I$ is an ideal on $X$ then $\I^\star$ is a filter (i.e., a family closed under supersets and finite intersections of its elements). We call it the dual filter of $\I$. Observe that the map $A\mapsto X\setminus A$ is continuous. Thus, $\I$ and $\I^\star$ have the same descriptive complexity. The above observation enables us to transfer the results of \cite{Debs} into the language of ideals.  

If $\I$ and $\J$ are ideals, then:
\begin{itemize} 
\item we say that $\J$ and $\I$ are isomorphic, if there is a bijection $f:\text{dom}(\J)\to \text{dom}(\I)$ such that 
$$\forall_{A\subseteq\text{dom}(\I)}\ f^{-1}[A]\in\J\ \Longleftrightarrow\ A\in\I;$$
\item we say that $\J$ contains an isomorphic copy of $\I$ and write $\I\sqsubseteq\J$, if there is a bijection $f:\text{dom}(\J)\to \text{dom}(\I)$ such that $f^{-1}[A]\in\J$ for each $A\in\I$;
\item we say that $\J$ is above $\I$ in the Kat\v{e}tov order and write $\I\leq_K\J$, if there is $f:\text{dom}(\J)\to \text{dom}(\I)$ (not necessary bijection) such that $f^{-1}[A]\in\J$ for each $A\in\I$.
\end{itemize}

We say that an ideal $\I$ on $X$ is generated by the family $\mathcal{F}\subseteq\mathcal{P}(X)$ if 
$$\I=\{A\subseteq X:\ \exists_{n\in\omega}\ \exists_{F_0,\ldots,F_n\in\mathcal{F}}\ A\subseteq F_0\cup\ldots\cup F_n\}.$$
If $Y\notin\I$ then the restriction of $\I$ to $Y$ is an ideal on $Y$ given by $\I|Y=\{A\cap Y:\ A\in\I\}$.

\subsection{Fubini sums of ideals}

If $\left(X_i\right)_{i\in I}$ is a family of sets, then by $\sum_{i\in I}X_i$ we denote its disjoint sum, i.e., the set of all pairs $(i,x)$ where $i\in I$ and $x\in X_i$.

Following \cite{Debs}, if $\I$ and $\J_i$, for $i\in \text{dom}(\I)$, are ideals, then the $\I$-Fubini sum of $(\J_i)_{i\in \text{dom}(\I)}$ is an ideal on $\sum_{i\in \text{dom}(\I)}\text{dom}(\J_i)$ defined by:
$$\I-\sum_{i\in \text{dom}(\I)}\J_i=\left\{\sum_{i\in \text{dom}(\I)}A_i:\ \left\{i\in \text{dom}(\I):\ A_i\notin\J_i\right\}\in\I\right\}.$$

If $\J_i=\J$ for all $i\in \text{dom}(\I)$ then $\I\otimes\J=\I-\sum_{i\in \text{dom}(\I)}\J_i$ is an ideal on $\text{dom}(\I)\times\text{dom}(\J)$ called the Fubini product of ideal $\I$ and $\J$. 

In \cite[p. 240]{Katetov} M. Kat\v{e}tov defined ideals $\fin^\alpha$, for $0<\alpha<\omega_1$ by:
\begin{itemize}
\item $\fin^1=\fin$, $\text{dom}(\fin^1)=\omega$;
\item $\fin^{\alpha+1}=\fin\otimes\fin^\alpha$, $\text{dom}(\fin^{\alpha+1})=\sum_{i\in\omega}\text{dom}(\fin^\alpha)$;
\item if $0<\lambda<\omega_1$ is a limit ordinal then $\fin^\lambda=\I_\lambda-\sum_{0<\alpha<\lambda}\fin^\alpha$ and $\text{dom}(\fin^{\lambda})=\sum_{0<\alpha<\lambda}\text{dom}(\fin^\alpha)$, where $\I_\lambda=\{A\subseteq\lambda\setminus\{0\}:\ \exists_{\alpha<\lambda}\ A\subseteq\alpha\}$ is an ideal on $\lambda\setminus\{0\}$;
\end{itemize}
(actually, paper \cite{Katetov} concerns filters instead of ideals). In particular, if $n\in\omega$ then $\fin^{n+1}=\fin\otimes\fin^n$ is an ideal on $\omega^{n+1}$. The ideals $\fin^\alpha$ will be useful in Subsection \ref{subsec} where we define ideals $\fin_\alpha$ (note the role of subscript and superscript in this notation).

\subsection{Inductive limits of ideals}

This subsection is a short summary of notions defined in \cite[Section 5]{Debs}, which we need to introduce in order to define the ideal $\fin_\omega$.

Following \cite[Definition 5.1]{Debs}, a quasi-homomorphism from $\mathcal{J}$ to $\mathcal{I}$ is a mapping $f:M\to {\rm dom}\left(\mathcal{I}\right)$, where $M\in\mathcal{J}^\star$, such that for each $A\in\mathcal{I}$ its preimage $f^{-1}[A]$ belongs to $\mathcal{J}$. It is easy to see that $\I\leq_K\J$ is equivalent to existence of a quasi-homomorphism from $\mathcal{J}$ to $\mathcal{I}$.

For a directed set $\left(I, \leq\right)$ and a family of ideals $\left(\mathcal{I}_i\right)_{i\in I}$, the family $\left(\pi_{i,j}\right)_{\stackrel{i\leq j}{i,j\in I}}$, where each $\pi_{i,j}$ is a quasi-homomorphism from $\mathcal{I}_j$ to $\mathcal{I}_i$, is called a system of quasi-homomorphisms for $\left(\mathcal{I}_i\right)_{i\in I}$.

\begin{definition}[see {\cite[Proposition-Definition 5.4]{Debs}}]
Let $\left(\pi_{i,j}\right)_{\stackrel{i\leq j}{i,j\in I}}$ be a system of quasi-homomorphisms for $\left(\mathcal{I}_i\right)_{i\in I}$. The set 
$$\underleftarrow{\lim}\left(\mathcal{I}_i, \pi_{i,j}\right)_{\stackrel{i\leq j}{i,j\in I}}=\left\{M\subseteq\sum_{i\in I}\text{dom}\left(\mathcal{I}_i\right):\ \exists_{i\in I}\ \exists_{P\in\mathcal{I}_i^\star}\ M\cap\left(\sum_{j\geq i} \pi_{i,j}^{-1}\left[P\right]\right)=\emptyset\right\}$$
is called the inductive limit of $\left(\mathcal{I}_i, \pi_{i,j}\right)_{\stackrel{i\leq j}{i,j\in I}}$. For simplicity, we write just $\underleftarrow{\lim}\I_i$ instead of $\underleftarrow{\lim}\left(\mathcal{I}_i, \pi_{i,j}\right)_{\stackrel{i\leq j}{i,j\in I}}$ whenever the system of quasi-homomorphisms is evident from the context.
\end{definition}

Inductive limits of ideals were introduced by G. Debs and J. Saint Raymond in order to define ideals $\fin_\alpha$ for limit ordinals $0<\alpha<\omega_1$ (see Subsection \ref{subsec}).

\subsection{Canonical ideals}
\label{subsec}

In this subsection we define ideals $\fin_\alpha$ for successor ordinals $0<\alpha<\omega_1$ and for $\alpha=\omega$. 

Observe that $\left(\pi_{i,j}\right)_{0<i\leq j<\omega}$, where $\text{dom}(\pi_{i,j})=\text{dom}(\fin^j)=\omega^j$ and $\pi_{i,j}$ is the projection onto the last $i$ coordinates, i.e., 
$$\pi_{i,j}(x_0,x_1,\ldots,x_{j-1})=(x_{j-i},x_{j-i+1},\ldots,x_{j-1}),$$ 
for all $(x_0,x_1,\ldots,x_{j-1})\in\omega^j$), is a system of quasi-homomorphisms for $\left(\fin^i\right)_{i\in \omega}$.

Following \cite[Section 6.2]{Debs}, we define:
\begin{itemize}
\item $\fin_1=\fin^1=\fin$;
\item $\fin_{\alpha+1}=\fin^{1+\alpha}$;
\item $\fin_\omega=\underleftarrow{\lim}(\fin^{n})_{n\in\omega}$.
\end{itemize}
Obviously, $\fin_n=\fin^n$ for all $n\in\omega$ and $\fin_{\alpha+1}=\fin^\alpha$ for all $\omega\leq\alpha<\omega_1$.

We do not introduce ideals $\fin_\lambda$ for limit ordinals $\omega<\lambda<\omega_1$. Except the facts that they are complicated objects and we will not use them in our considerations, the reason for omitting their definition is that those ideals were defined ambiguously in \cite[Section 6.2]{Debs}. Actually, that is also the case of $\fin_\omega$. In \cite[Section 6.2]{Debs} it is only written that since each $\fin^{n+1}$ is obtained as the $\fin$-Fubini sum of $\fin^{n}$, applying \cite[Proposition 5.8]{Debs} inductively one gets a system of quasi-homomorphisms
for $(\fin^n)_{n\in\omega}$. While from \cite[Proposition 5.8]{Debs} (and its proof) we see that any quasi-inductive system of ideals $\left(\mathcal{I}_i, \nu_{i,j}\right)_{\stackrel{i\leq j}{i,j\in\omega}}$ can be extended to $\left(\mathcal{I}_i, \nu_{i,j}\right)_{\stackrel{i\leq j}{i,j\in \omega+1}}$, where $\I_{\omega}$ is the $\fin$-Fubini sum of $(\I_i)_{i\in \omega}$ and $\nu_{i,\omega}=\sum_{j\geq i,j\in \omega}\nu_{i,j}$ is the disjoint union mapping. Thus, in order to get a quasi-homomorphism $\pi_{1,2}$ from $\fin^2$ to $\fin$ we need to put $\I_i=\fin$ and can guess that $\pi_{i,j}(x)=x$, for all $i,j\in\omega$, $i\leq j$ and $x\in\omega$. However, the problem is that applying \cite[Proposition 5.8]{Debs} we get infinitely many quasi-homomorphism from $\fin^2$ to $\fin$ -- each of them is given by the projection onto the second coordinate, but their domains differ (and are of the form $(\omega\setminus i)\times\omega$). Nothing is mentioned about which of those quasi-homomorphism we should choose. In the above definition of $\fin_\omega$ we always choose the quasi-homomorphism with the largest domain (which seems to be the most obvious choice). Fortunately, regardless of those issues, it seems that each version of $\fin_\lambda$, for a limit ordinals $0<\lambda<\omega_1$, possesses all the properties required in \cite{Debs} (provided that it is an ideal -- see the next section).

\section{When an inductive limit of ideals is an ideal?}

In this section we show that $\fin_\omega$ is an ideal despite the fact that \cite[Proposition-Definition 5.4]{Debs}, which defines it, is false. 

\begin{definition}[{\cite[Definition 5.3]{Debs}}]
A system of quasi-homomorphisms $\left(\pi_{i,j}\right)_{i\leq j}$ for a family of ideals $\left(\mathcal{I}_i\right)_{i\in I}$ is coherent provided that for all $i,j,k\in I$ with $i\leq j\leq k$ the following condition is satisfied:
$$\pi_{i,k}(a)=\pi_{i,j}\left(\pi_{j,k}(a)\right)$$
for every $a\in\text{dom}\left(\pi_{i,k}\right)\cap\text{dom}\left(\pi_{j,k}\right)\cap\pi_{j,k}^{-1}\left(\text{dom}\left(\pi_{i,j}\right)\right)$. In this case we call $\left(\mathcal{I}_i, \pi_{i,j}\right)_{\stackrel{i\leq j}{i,j\in I}}$ a coherent quasi-inductive system of ideals.
\end{definition}

Next example shows that an inductive limit of a coherent quasi-inductive system of ideals does not have to be an ideal. This means that \cite[Proposition-Definition 5.4]{Debs} is false.

\begin{example}
\label{ExIndLim}
Put $I=\omega\setminus\{0\}$ and consider the ideals $\fin^i$ for $i\in I$ (i.e., $\fin^1=\fin$, $\fin^2=\fin\otimes\fin$ and $\fin^{i+1}=\fin\otimes\fin^i$). For each $i>1$ let $F_i=\omega^i\setminus\left(\left\{0\right\}\times\omega^{i-1}\right)\in(\fin^i)^\star$ and define a quasi-homomorphism $\pi_{1,i}:F_i\to\omega$ from $\fin^i$ to $(\fin^1)^\star$ by $\pi_{1,i}(x_0,\ldots,x_{i-1})=x_{i-1}$ for all $(x_0,\ldots,x_{i-1})\in F_i$ (i.e., $\pi_{1,i}$ is the projection onto the last coordinate). For indices $1<i\leq j$ let $\pi_{i,j}:\omega^j\to\omega^i$ be the projection onto the last $i$ coordinates. Then $\left(\pi_{i,j}\right)_{i\leq j}$ is a coherent system of quasi-homomorphisms for the family of ideals $\left(\fin^i\right)_{i\in I}$. However, we will show that its inductive limit, let us denote it by $\mathcal{J}$, is not an ideal.

Consider $A=\sum_{j\geq 1}\pi_{1,j}^{-1}[\omega]$ and $B=\sum_{j\geq 2}\pi_{2,j}^{-1}[\omega^2]$. Obviously, $A^c,B^c\in\mathcal{J}$. We claim that 
$$A^c\cup B^c=(A\cap B)^c=\left(\Sigma_{j>1} F_j\right)^c$$
does not belong to $\mathcal{J}$. Indeed, if $i>1$ and $P\in(\fin^i)^\star$, then $\sum_{j\geq i}\pi_{i,j}^{-1}[P]\not\subseteq \Sigma_{j>1} F_j$ as for any $(x_0,\ldots,x_{i-1})\in P\subseteq\omega^i$ we have $\pi_{i,i+1}((0,x_0,\ldots,x_{i-1}))\in P$ and $(0,x_0,\ldots,x_{i-1})\notin F_{i+1}$. On the other hand, if $i=1$ then $\sum_{j\geq i}\pi_{i,j}^{-1}[P]\not\subseteq A\cap B$ as $(\{1\}\times P)\subseteq \sum_{j\geq i}\pi_{i,j}^{-1}[P]$ and $(\{1\}\times P)\cap \left(\Sigma_{j>1} F_j\right)=\emptyset$. This shows that $\mathcal{J}$ is not an ideal.
\end{example}

Now we want to focus on showing that $\fin_\omega$ is an ideal. We will need the following notion.

\begin{definition}
We say that a quasi-inductive system $\left(\mathcal{I}_i, \pi_{i,j}\right)_{\stackrel{i\leq j}{i,j\in I}}$ satisfies condition (C) if it is coherent and
$$\forall_{k\geq j}\ \pi^{-1}_{j,k}\left({\rm dom}(\pi_{i,j})\right)\subseteq{\rm dom}(\pi_{i,k})$$
for all indices $i\leq j$. 
\end{definition}

\begin{proposition}
Let $\left(\mathcal{I}_i, \pi_{i,j}\right)_{\stackrel{i\leq j}{i,j\in I}}$ be a quasi-inductive system satisfying condition (C). Then $\underleftarrow{\lim}\mathcal{I}_i$ is an ideal.
\end{proposition}

\begin{proof}
Define ideals
$$\overline{\mathcal{I}}_i=\left\{M\subseteq\Sigma_{k\in I}{\rm dom}\left(\mathcal{I}_k\right):\ \exists_{P\in\mathcal{I}^\star_i}\ M\cap\overline{P}=\emptyset\right\},$$
where for $i\in I$ and $P\in\mathcal{I}^\star_i$ we denote $\overline{P}=\sum_{j\geq i}\pi_{i,j}^{-1}[P]$. 

Observe that
$$\underleftarrow{\lim}\mathcal{I}_i=\bigcup_{i\in I}\overline{\mathcal{I}}_i.$$
We will show that $\overline{\mathcal{I}}_i\subseteq\overline{\mathcal{I}}_j$ provided that $i\leq j$. It will follow that $\underleftarrow{\lim}\mathcal{I}_i$ is an ideal as a union of an increasing family of ideals.

To show that $\overline{\mathcal{I}}_i\subseteq\overline{\mathcal{I}}_j$ whenever $i\leq j$, fix $M\in\overline{\mathcal{I}}_i$. Then there is $P\in\mathcal{I}^\star_i$ such that $M\cap\overline{P}=\emptyset$. Set $P'=\pi_{i,j}^{-1}\left[P\right]\in\mathcal{I}^\star_j$. We will show that $\overline{P}'\subseteq\overline{P}$ (which will end the proof as $\overline{P}'\cap M\subseteq\overline{P}\cap M=\emptyset$). If $k\geq j$ and $x\in\pi_{j,k}^{-1}\left[P'\right]$ (so $(k,x)$ is an element of $\overline{P}'$) then $x\in\pi^{-1}_{j,k}\left({\rm dom}(\pi_{i,j})\right)$ (as $P'\subseteq{\rm dom}(\pi_{i,j})$). By condition (C) we know that $x\in{\rm dom}(\pi_{i,k})$. Now, using the fact that $\left(\mathcal{I}_i, \pi_{i,j}\right)_{\stackrel{i\leq j}{i,j\in I}}$ is coherent, we get $\pi_{i,k}(x)=\pi_{i,j}(\pi_{j,k}(x))\in P$. Thus, $x$ is an element of $\pi_{i,k}^{-1}\left[P\right]$ and $(k,x)$ is an element of $\overline{P}$. This finishes the proof.
\end{proof}

\begin{corollary}
$\fin_\omega$ is an ideal.
\end{corollary}

\begin{proof}
It is easy to show that $\left(\pi_{i,j}\right)_{0<i\leq j<\omega}$ defined in Subsection \ref{subsec} is a system of quasi-homomorphisms for $\left(\fin^i\right)_{i\in \omega}$ satisfying condition (C). Hence, the previous result implies that $\fin_\omega$ is an ideal.
\end{proof}

We end this section with a modification of \cite[Proposition 5.8]{Debs}.

\begin{proposition}
\label{C}
Let $\left(\mathcal{I}_i, \pi_{i,j}\right)_{\stackrel{i\leq j}{i,j\in I}}$ be a quasi-inductive system satisfying condition (C), $\J$ be the ideal on $I$ generated by the family $\{\{j\in I:\ j\leq i\}:\ i\in I\}$, and $\mathcal{I}_\infty$ denote the $\J$-Fubini sum of $\left(\mathcal{I}_i\right)_{i\in I}$. If $\pi_{i,\infty}:\sum_{k\geq i}{\rm dom}(\pi_{i,k})\to{\rm dom}(\I_i)$ are given by $\pi_{i,\infty}=\sum_{j\geq i}\pi_{i,j}$, then $\left(\mathcal{I}_i, \pi_{i,j}\right)_{\stackrel{i\leq j}{i,j\in I\cup\left\{\infty\right\}}}$ is also a quasi-inductive system satisfying condition (C).
\end{proposition}

\begin{proof}
In \cite[Proposition 5.8]{Debs} it is shown that $\left(\mathcal{I}_i, \pi_{i,j}\right)_{\stackrel{i\leq j}{i,j\in I\cup\left\{\infty\right\}}}$ is a coherent system. We need to show that 
$$\forall_{k\geq j}\ \pi^{-1}_{j,k}\left({\rm dom}(\pi_{i,j})\right)\subseteq{\rm dom}(\pi_{i,k})$$
for all indices $i\leq j$. 

Set $i\leq j$, $i,j\in I$. Observe that 
$$\pi^{-1}_{j,\infty}\left({\rm dom}(\pi_{i,j})\right)=\sum_{k\geq j}\pi^{-1}_{j,k}\left({\rm dom}(\pi_{i,j})\right).$$
Applying condition (C) we obtain that
$$\sum_{k\geq j}\pi^{-1}_{j,k}\left({\rm dom}(\pi_{i,j})\right)\subseteq\sum_{k\geq j}{\rm dom}(\pi_{i,k})\subseteq\sum_{k\geq i}{\rm dom}(\pi_{i,k})={\rm dom}(\pi_{i,\infty}),$$
which finishes the proof.
\end{proof}

\section{The counterexample}

This section is devoted to showing that Conjecture \ref{hip} is false in the case of $\alpha=\omega$. Our counterexample is defined as follows.

\begin{definition}
\label{maindef}
Fix a family $\{\{X_s:s\in\omega^{n+1}\}:n\in\omega\}$ of partitions of $\omega$ into infinite sets (i.e., for every $n\in\omega$ each $X_s$, for $s\in\omega^{n+1}$, is infinite and $\{X_s:s\in\omega^{n+1}\}$ is a partition of $\omega$) such that $\bigcap_{i\in F}X_{s_i}$ is infinite for all $F\in[\omega]^{<\omega}$ and $s_i\in\omega^{i+1}$, $i\in F$. Then $\fin'_\omega$ is an ideal on $\omega$ generated by the family $\bigcup_{n\in\omega}\fin^{n+2}(\{X_s:s\in\omega^{n+1}\})$, where $\fin^{n+2}(\{X_s:s\in\omega^{n+1}\})$ is an isomorphic copy of $\fin^{n+2}$ on $\omega$ given by any bijection $\sigma:\omega^{n+2}\to\omega$ such that $\sigma[\{s\}\times\omega]=X_s$ for all $s\in\omega^{n+1}$ (note that each such bijection defines the same ideal on $\omega$).
\end{definition}

Before proving that $\fin'_\omega$ is as needed, we need to establish some of its properties which will be helpful in our considerations.

For a family $\mathcal{A}\subseteq\mathcal{P}(X)$ we denote $\mathcal{A}^\uparrow=\{B\subseteq X: \exists_{A\in\mathcal{A}}B\supseteq A\}$.

\begin{definition}
Define 
$$\J_n=\left\{A\subseteq\omega:\ \phi_n(A)\in\left(\fin^\star\times\ldots\times(\fin^{n+1})^\star\right)^\uparrow\right\},$$
where $\phi_n:\mathcal{P}(\omega)\to\mathcal{P}(\omega\times\omega^{2}\times\ldots\times\omega^{n+1})$ is given by 
$$\phi_n(A)=\left\{(s_0,\ldots,s_n)\in\omega\times\omega^2\times\ldots\times\omega^{n+1}:\ A\cap X_{s_0}\cap\ldots\cap X_{s_n}\in\fin\right\}.$$
\end{definition}

\begin{lemma}
\label{lem}
$\fin'_\omega=\bigcup_{n\in\omega}\J_n$ and $\J_n\subseteq\J_{n+1}$, for each $n\in\omega$.
\end{lemma}

\begin{proof}
The inclusion $\J_n\subseteq\J_{n+1}$, for each $n\in\omega$, is obvious. Thus, we only prove $\fin'_\omega=\bigcup_{n\in\omega}\J_n$.

Let $A\in\fin'_\omega$. Then there are $n\in\omega$ and sets $A_i$, for $i\leq n$, such that $A=\bigcup_{i\leq n}A_i$ and $A_i\in\fin^{i+2}(\{X_s:s\in\omega^{i+1}\})$, for all $i\leq n$. Thus, for each $i\leq n$ there is $B_i\in\fin^{i+1}$ such that $A_i\cap X_s\notin\fin$ if and only if $s\in B_i$, for all $s\in \omega^{i+1}$. Observe that $(B_0)^c\times\ldots\times (B_n)^c\subseteq\phi_n(A)$, since for each $(s_0,\ldots,s_n)\in (B_0)^c\times\ldots\times (B_n)^c$ we have 
$$A\cap X_{s_0}\cap\ldots\cap X_{s_n}=\bigcup_{i\leq n}\left(A_i\cap X_{s_0}\cap\ldots\cap X_{s_n}\right)\subseteq\bigcup_{i\leq n}A_i\cap X_{s_i}\in\fin,$$
i.e., $(s_0,\ldots,s_n)\in\phi_n(A)$. As $(B_0)^c\times\ldots\times (B_n)^c\in \fin^\star\times\ldots\times(\fin^{n+1})^\star$ and $(\fin^\star\times\ldots\times(\fin^{n+1})^\star)^\uparrow$ is closed under supersets, we conclude that $\phi_n(A)\in(\fin^\star\times\ldots\times(\fin^{n+1})^\star)^\uparrow$.

Let $n\in\omega$ and $A\subseteq\omega$ be such that $\phi_n(A)\in(\fin^\star\times\ldots\times(\fin^{n+1})^\star)^\uparrow$. Then there are $B_0,\ldots,B_n$ such that $\phi_n(A)\supseteq(B_0)^c\times\ldots\times (B_n)^c$ and $B_i\in\fin^{i+1}$, for all $i\leq n$. Define $A_i=\bigcup_{s\in B_i}X_s$ for all $i\leq n$. Obviously, $A_i\in\fin^{i+2}(\{X_s:s\in\omega^{i+1}\})$, for all $i\leq n$, i.e., $A'=\bigcup_{i\leq n}A_i\in\fin'_\omega$. Moreover, $(A\setminus A')\cap X_{s_0}\cap\ldots\cap X_{s_n}\in\fin$ for all $(s_0,\ldots,s_n)\in \omega\times\ldots\times\omega^{n+1}$. 

To finish the proof, we will inductively show that for each $B\subseteq\omega$ and $n\in\omega$, if $B\cap X_{s_0}\cap\ldots\cap X_{s_n}\in\fin$ for all $(s_0,\ldots,s_n)\in \omega\times\ldots\times\omega^{n+1}$, then $B\in\fin'_\omega$.

For $n=0$ this is obvious as $B\cap X_{s_0}\in\fin$ for all $s_0\in \omega$ implies $B\in\fin^{2}(\{X_s:s\in\omega\})\subseteq\fin'_\omega$.

Assume that the induction hypothesis is true for all $k<n$ and fix $B\subseteq\omega$ with $B\cap X_{s_0}\cap\ldots\cap X_{s_n}\in\fin$ for all $(s_0,\ldots,s_n)\in \omega\times\ldots\times\omega^{n+1}$. Fix two bijections $h:\omega\to\omega^{n+1}$ and $g:\omega\to\omega\times\ldots\times\omega^{n}$ and for each $(s_0,\ldots,s_{n-1})\in \omega\times\ldots\times\omega^{n}$ denote 
$$Y_{(s_0,\ldots,s_{n-1})}=X_{s_0}\cap\ldots\cap X_{s_{n-1}}$$
(note that $(Y_{g(j)})_{j\in\omega}$ is a partition of $\omega$). Define $C=\bigcup_{k\in\omega}\bigcup_{j\leq k}B\cap X_{h(k)}\cap Y_{g(j)}$ and $D=\bigcup_{j\in\omega}\bigcup_{k\leq j}B\cap X_{h(k)}\cap Y_{g(j)}$. Then $D$ is in $\fin'_\omega$ by the induction hypothesis, as $D\cap Y_{g(j)}\in\fin$ for all $j\in\omega$. Similarly, $C\in\fin^{n+2}(\{X_s:s\in\omega^{n+1}\})$ as $C\cap X_{h(k)}\in\fin$ for all $k\in\omega$. Finally, $B=C\cup D$ as for each $x\in B$ we can find a unique $(s_0,\ldots,s_n)\in \omega\times\ldots\times\omega^{n+1}$ with $x\in X_{s_0}\cap\ldots\cap X_{s_n}$ and either $h^{-1}(s_n)\leq g^{-1}((s_0,\ldots,s_{n-1}))$ (in this case $x\in D$) or $h^{-1}(s_n)\geq g^{-1}((s_0,\ldots,s_{n-1}))$ (in this case $x\in C$).
\end{proof}

\begin{lemma}
\label{Borelcomplexity}
Let $n\in\omega$. Then $\J_n$ is a $\bf{\Sigma^0_{2n+4}}$ ideal. In particular, $\rk(\J_n)\leq 2n+4$.
\end{lemma}

\begin{proof}
It is easy to see that $\J_n$ is an ideal. The "'In particular"' part will follow directly from $\J_n\in\bf{\Sigma^0_{2n+4}}$ and the definition of rank of an ideal. Thus, we only need to show that $\J_n\in\bf{\Sigma^0_{2n+4}}$.

For this proof, for a given $n\in\omega$ denote by $\mathcal{P}_n$ the family of all $\mathcal{Y}=\{\{Y_s:s\in\omega^{i+1}\}:i\leq n\}$ such that:
\begin{itemize}
\item $Y_s$ is a subset of $\omega$, for each $s\in\omega^{i+1}$, $i\leq n$;
\item $Y_s\cap Y_t=\emptyset$ whenever $s,t\in\omega^{i+1}$ for some $i\leq n$ and $s\neq t$.
\end{itemize}
Moreover, for given $n\in\omega$ and $\mathcal{Y}\in\mathcal{P}_n$ we denote by $\J_n(\mathcal{Y})$ the family consisting of all $A\subseteq\omega$ such that $\phi_n(\mathcal{Y})(A)\in\left(\fin^\star\times\ldots\times(\fin^{n+1})^\star\right)^\uparrow$, where 
$$\phi_n(\mathcal{Y})(A)=\left\{(s_0,\ldots,s_n)\in\omega\times\omega^2\times\ldots\times\omega^{n+1}: A\cap Y_{s_0}\cap\ldots\cap Y_{s_n}\in\fin\right\}.$$

We will inductively show that for every $n\in\omega$ and $\mathcal{Y}\in\mathcal{P}_n$ the family $\J_n(\mathcal{Y})$ is $\bf{\Sigma^0_{2n+4}}$. This will finish the proof since for each $n\in\omega$ we have $\J_n=\J_n(\{\{X_s:s\in\omega^{i+1}\}:i\leq n\})$, where the sets $X_s$, for $s\in\omega^{i+1}$, $i\leq n$, are as in Definition \ref{maindef}.

Firstly, fix any $\mathcal{Y}=\{\{Y_s:s\in\omega\}\}\in\mathcal{P}_0$ and observe that $\J_0(\mathcal{Y})$ is $\bf{\Sigma^0_{4}}$, since 
$$\J_0(\mathcal{Y})=\{A\subseteq\omega:\ \exists_{i\in\omega}\ \forall_{j>i}\ A\cap Y_j\in\fin\}=\bigcup_{i\in\omega}\bigcap_{j>i}\bigcup_{F\in[Y_i]^{<\omega}}\{A\subseteq\omega:\ A\cap Y_i=F\}$$
and $\{A\subseteq\omega:\ A\cap Y_i=F\}$ is closed.

Now, fix any $\mathcal{Y}=\{\{Y_s:s\in\omega^{i+1}\}:i\leq 1\}\in\mathcal{P}_1$ and observe that:
$$\J_1(\mathcal{Y})=\left\{A\subseteq\omega:\ \exists_{i\in\omega}\ \forall_{j>i}\ \exists_{k\in\omega}\ \forall_{l>k}\ A\cap Y_{j}\cap Y_{(j,l)}\in\fin\right\}=$$
$$=\left\{A\subseteq\omega:\ \exists_{i\in\omega}\ \forall_{j>i}\ A\cap Y_j\in\J_0((Y_{\{j\}\times s})_{s\in\omega})\right\}.$$
Since $A\mapsto A\cap X_j$ is continuous and $\J_0((Y_{\{j\}\times s})_{s\in\omega})\in\bf{\Sigma^0_{4}}$, for each $j\in\omega$, the family $\J_1(\mathcal{Y})$ is $\bf{\Sigma^0_{6}}$.

Thus, we have shown the thesis for $n=0$ and $n=1$. Suppose now that $\J_m(\mathcal{Y})$ is $\bf{\Sigma^0_{2m+4}}$ for all $m<n$ and $\mathcal{Y}\in\mathcal{P}_m$. Fix any $\mathcal{Y}=\{\{Y_s:s\in\omega^{i+1}\}:i\leq n\}\in\mathcal{P}_n$. Note that:
$$\J_n(\mathcal{Y})=\left\{A\subseteq\omega:\ \exists_{i\in\omega}\ \forall_{j>i}\ A\cap Y_j\in\J_{n-1}((Y_{\{j\}\times s})_{s\in\omega},\ldots,(Y_{\{j\}\times s})_{s\in\omega^n})\right\}.$$
Hence, using the inductive assumption, we get that $\J_n(\mathcal{Y})$ is $\bf{\Sigma^0_{2n+4}}$.
\end{proof}

We are ready to prove that $\fin'_\omega$ is as needed, i.e., that it is a Borel ideal of rank $\omega$ not containing any isomorphic copy of $\fin_\omega$.

\begin{theorem}
\label{x}
$\fin'_\omega$ is a $\bf{\Sigma^0_\omega}$ ideal of rank $\omega$.
\end{theorem}

\begin{proof}
As $\fin_{n+2}\sqsubseteq\fin'_\omega$ for all $n\in\omega$, $\text{rk}(\fin'_\omega)\geq n+2$ for all $n\in\omega$ (by \cite[Lemma 7.2]{Debs}), i.e., $\text{rk}(\fin'_\omega)\geq \omega$. On the other hand, $\text{rk}(\fin'_\omega)\leq \omega$ since $\fin'_\omega$ is $\bf{\Sigma^0_\omega}$ (by Lemmas \ref{lem} and \ref{Borelcomplexity}). 
\end{proof}

\begin{theorem}
$\fin_\omega\not\sqsubseteq\fin'_\omega$.
\end{theorem}

\begin{proof}
Fix any bijection $f:\omega\to\sum_{i>0}\omega^i$ and suppose to the contrary that it witnesses $\fin_\omega\sqsubseteq\fin'_\omega$. For a set $X\subseteq\omega^k$ denote $\overline{X}=\sum_{i\geq k}\pi_{k,i}^{-1}[X]$ (so that $X\in\fin^k$ implies $\overline{X}\in\fin_\omega$ and $f^{-1}[\overline{X}]\in\fin'_\omega$).

Since $\{1\}\times\omega,\{2\}\times\omega^2\in\fin_\omega$, we have $A=f^{-1}[(\{1\}\times\omega)\cup(\{2\}\times\omega^2)]\in\fin'_\omega$. For each $i,j\in\omega$ the set $\{(i,j)\}\times\omega$ is in $\fin^3$, so $A_{i,j}=f^{-1}\left[\overline{\{(i,j)\}\times\omega}\right]\in\fin'_\omega$. Observe that $\omega=A\cup\bigcup_{i,j\in\omega}A_{i,j}$. 

We claim that there is $m\in\omega$ such that each $Y\subseteq\omega\setminus A$ with
$$\exists_{h\in\omega^\omega}\ Y\subseteq\bigcup_{i\in\omega}\bigcup_{j\leq h(i)}A_{i,j}$$
is in $\J_m$. Suppose otherwise. Then for each $k\in\omega$ we could find $Y_k\subseteq\omega\setminus A$ such that $Y_k\subseteq\bigcup_{i\in\omega}\bigcup_{j\leq h_k(i)}A_{i,j}$ for some $h_k\in\omega^\omega$, but $Y_k\notin\J_k$. Note that $B_i=\bigcup_{j\in\omega}A_{i,j}\in\fin'_\omega$ for each $i\in\omega$, as $B_i=f^{-1}\left[\overline{\{i\}\times\omega^2}\right]$ and $\{i\}\times\omega^2$ is in $\fin^3$. Let $g\in\omega\to\omega$ be any nondecreasing function satisfying $\lim_i g(i)=+\infty$ and $\bigcup_{j\leq g(i)}B_{j}\in\J_{i}$ for all $i\in\omega$ (recall that $\J_i\subseteq\J_{i+1}$, for all $i\in\omega$, by Lemma \ref{lem}, so such function exists). Define $Y=\bigcup_{k\in\omega}Y_k\setminus(B_0\cup\ldots\cup B_{g(k)})$. Obviously, $Y\notin\fin'_\omega=\bigcup_{k\in\omega} \J_k$ as for each $k\in\omega$ we have $Y_k\setminus(B_0\cup\ldots\cup B_{g(k)})\subseteq Y$, $Y_k\notin\J_k$ and $B_0\cup\ldots\cup B_{g(k)}\in\J_k$. On the other hand, denoting $a_i=\max\{h_k(i):k\in g^{-1}(\{0,\ldots,i-1\})\}$ (note that $g^{-1}(\{0,\ldots,1\})$ is finite by $\lim_i g(i)=+\infty$), we have:
$$Y=\bigcup_{i\in\omega}(B_i\cap\bigcup\{Y_k:g(k)<i\})\subseteq\bigcup_{i\in\omega}\bigcup_{j\leq a_i}A_{i,j}=f^{-1}\left[\overline{\bigcup_{i\in\omega}\bigcup_{j\leq a_i}\{(i,j)\}\times\omega}\right].$$
Since $\bigcup_{i\in\omega}\bigcup_{j\leq a_i}\{(i,j)\}\times\omega\in\fin^3$, we get that $Y\in\fin'_\omega$. This contradiction proves that there is $m\in\omega$ such that each $Y\subseteq\omega\setminus A$ satisfying $Y\subseteq\bigcup_{i\in\omega}\bigcup_{j\leq h(i)}A_{i,j}$, for some $h\in\omega^\omega$, is in $\J_m$. 

We need to introduce some notations. Define $A'=f^{-1}[\bigcup_{i\leq 2m+4}\{i\}\times\omega^i]$ and note that $A'\in\fin'_\omega$ as $\bigcup_{i\leq 2m+4}\{i\}\times\omega^i\in\fin_\omega$. Let $C^n_s=f^{-1}\left[\overline{\{s\}\times\omega}\right]\setminus A'$ for all $n\leq 2m+4$ and $s\in\omega^{n}$ (i.e., $C^2_{(i,j)}=A_{i,j}\setminus A'$). Consider the following ideals on $\omega\setminus A'$:
$$Y\in\overline{\fin^{2m+5}}\ \Longleftrightarrow\ \{s\in\omega^{2m+4}: Y\cap C^{2m+4}_s\notin\fin\}\in\fin^{2m+4};$$
$$Y\in\overline{\fin^{2m+4}}\ \Longleftrightarrow\ \{s\in\omega^{2m+3}: Y\cap C^{2m+3}_s\notin\fin\}\in\fin^{2m+3};$$
$$Y\in\overline{\emptyset^2\otimes\fin^{2m+2}}\ \Longleftrightarrow\ \forall_{t\in\omega^2}\ \{s\in\omega^{2m+1}: Y\cap C^{2m+3}_{t^\frown s}\neq\emptyset\}\in\fin^{2m+1};$$
$$Y\in\overline{(\emptyset\otimes\fin)\otimes\emptyset}\ \Longleftrightarrow\ \{s\in\omega^{2}: Y\cap C^{2}_s\neq\emptyset\}\in\emptyset\otimes\fin;$$
$$Y\in\overline{\fin\otimes\emptyset}\ \Longleftrightarrow\ \{n\in\omega: Y\cap C^{1}_{(n)}\neq\emptyset\}\in\fin.$$
By the choice of $m$, we have $\overline{(\emptyset\otimes\fin)\otimes\emptyset}\subseteq\J_m$. Moreover, 
$$\overline{\emptyset^2\otimes\fin^{2m+2}}\cup\overline{(\emptyset\otimes\fin)\otimes\emptyset}\cup\overline{\fin\otimes\emptyset}\subseteq\overline{\fin^{2m+4}}\subseteq \overline{\fin^{2m+5}}$$
and for each $Y\in \overline{\fin^{2m+4}}$ there are $Y^0\in\overline{\emptyset^2\otimes\fin^{2m+2}}$, $Y^1\in\overline{(\emptyset\otimes\fin)\otimes\emptyset}$ and $Y^2\in\overline{\fin\otimes\emptyset}$ with $Y\subseteq Y^0\cup Y^1\cup Y^2$.

What is more, for each $n\in\omega$ consider the ideals $\overline{\fin^{2m+4}_n}$ and $\overline{\emptyset\otimes\fin^{2m+3}_n}$ on $C^{1}_{(n)}$ given by:
$$Y\in\overline{\fin^{2m+4}}(n)\ \Longleftrightarrow\ \{s\in\omega^{2m+3}: Y\cap C^{2m+4}_{(n)^\frown s}\notin\fin\}\in\fin^{2m+3};$$
$$Y\in\overline{\emptyset\otimes\fin^{2m+3}}(n)\ \Longleftrightarrow\ \{s\in\omega^{2m+3}: Y\cap C^{2m+4}_{(n)^\frown s}\notin\fin\}\in\emptyset\otimes\fin^{2m+2}_n.$$
Then $\overline{\emptyset\otimes\fin^{2m+3}}(n)\subseteq\overline{\fin^{2m+4}}(n)$ and if $Y_n\in\overline{\fin^{2m+4}}(n)$ for all $n\in\omega$ then $\bigcup_{n\in\omega}Y_n\in\overline{\fin^{2m+5}}$. Observe that $\overline{\fin^{2m+5}}$ is isomorphic with $\fin^{2m+5}$ and $\overline{\fin^{2m+4}}$ as well as each $\overline{\fin^{2m+4}}(n)$ are isomorphic with $\fin^{2m+4}$. 

We claim that $Y\in\overline{\fin^{2m+5}}$ implies $f[Y]\in\fin_\omega$. Indeed, fix $Y\in\overline{\fin^{2m+5}}$ (in particular, $Y$ is a subset of $\omega\setminus A'$), denote $S=\{s\in\omega^{2m+4}: Y\cap C^{2m+4}_s\notin\fin\}\in\fin^{2m+4}$ and observe that $f[\bigcup_{s\in S}C^{2m+4}_s]=\overline{S\times\omega}\in\fin_\omega$ as $S\in\fin^{2m+4}$ and $S\times\omega\in\fin^{2m+5}$. For each $x\in Y\setminus (\bigcup_{s\in S}C^{2m+4}_s)$ let $i_x\geq 2m+5$ be the unique index with $f(x)\in\{i_x\}\times\omega^{i_x}$. Then $|\{x\in Y\setminus (\bigcup_{s\in S}C^{2m+4}_s):\pi_{2m+5,i_x}(f(x))\in\{s\}\times\omega\}|<\omega$ for each $s\in\omega^{2m+4}$ and 
$$f\left[Y\setminus \left(\bigcup_{s\in S}C^{2m+4}_s\right)\right]\subseteq\overline{\left\{\pi_{2m+5,i_x}(f(x)): x\in Y\setminus \left(\bigcup_{s\in S}C^{2m+4}_s\right)\right\}}\in\fin_\omega.$$
Thus, $f[Y]\subseteq f[\bigcup_{s\in S}C^{2m+4}_s]\cup f[Y\setminus (\bigcup_{s\in S}C^{2m+4}_s)]\in\fin_\omega$.

By the above observation, in order to finish the proof it suffices to find $W\in\overline{\fin^{2m+5}}\setminus\fin'_\omega$. However, at first we need to define inductively sequences $(j(n))\subseteq\omega$ and $(Z_n)\subseteq\mathcal{P}(\omega)$ such that $Z_n\subseteq Y(h(n))$ and $Z_n\in\overline{\fin^{2m+5}}\setminus\J_m|(Y(h(n))$ for all $n\in\omega$, where $h:\omega\to\bigcup_{n>m+1}(\omega^{m+2}\times\ldots\times\omega^n)$ is a fixed bijection and 
$$Y(t)=\bigcup\{X_{s_0}\cap\ldots\cap X_{s_m}\cap X_{t_0}\cap\ldots\cap X_{t_k}:\ (s_0,\ldots,s_m)\in\omega\times\omega^2\times\ldots\times\omega^{m+1}\},$$ 
for all $k\in\omega$ and $t=(t_0,\ldots,t_k)\in\omega^{m+2}\times\omega^2\times\ldots\times\omega^{m+k+2}$ (note that $Y(t)\notin\J_m$ as $X_{s_0}\cap\ldots\cap X_{s_m}\cap X_{t_0}\cap\ldots\cap X_{t_k}$ are infinite). The sequence $(j(n))$ will indicate whether $Z_n\in\overline{\emptyset^2\otimes\fin^{2m+2}}$ (the case of $j(n)=0$) or $Z_n\in\overline{\emptyset\otimes\fin^{2m+3}}(k)$ (the case of $j(n)=k+1$).

Pick any $n\in\omega$ and consider the ideal $\J_m|Y(h(n))$. Recall that $\overline{\fin^{2m+4}}$ is isomorphic with $\fin^{2m+4}$, which is homogeneous (see \cite{Fremlin} or \cite[Remark after Proposition 2.9]{zJackiem}), i.e., either $Y(h(n))\in\overline{\fin^{2m+4}}$ or $\overline{\fin^{2m+4}}|Y(h(n))$ is isomorphic with $\fin^{2m+4}$. Moreover, as $\J_m$ is $\bf{\Sigma^0_{2m+4}}$ (by Lemma \ref{Borelcomplexity}) and the identity function from $\mathcal{P}(Y(h(n)))$ to $\mathcal{P}(\omega)$ is continuous, $\J_m|Y(h(n))$ is $\bf{\Sigma^0_{2m+4}}$ as well, so $\text{rk}(\J_m|Y(h(n)))\leq 2m+3$. Thus, either $Y(h(n))\in\overline{\fin^{2m+4}}$ or $\overline{\fin^{2m+4}}|Y(h(n))\not\subseteq\J_m|Y(h(n))$ (by \cite[Lemma 7.2]{Debs}). Let $Z\subseteq Y(h(n))$ be such that $Z\in\overline{\fin^{2m+4}}\setminus\J_m|Y(h(n))$. Then there are $Z^0\in\overline{\emptyset^2\otimes\fin^{2m+2}}$, $Z^1\in\overline{(\emptyset\otimes\fin)\otimes\emptyset}$ and $Z^2\in\overline{\fin\otimes\emptyset}$ with $Z\subseteq Z^0\cup Z^1\cup Z^2\subseteq Y(h(n))$ and $Z^0\cup Z^1\cup Z^2\in\overline{\fin^{2m+4}}$. Since $Z^1\in\J_m|Y(h(n))$ by $\overline{(\emptyset\otimes\fin)\otimes\emptyset}\subseteq\J_m$, there are two possible cases:
\begin{itemize}
\item If $Z^0\notin\J_m|Y(h(n))$ then put $j(n)=0$ and $Z_n=Z^0$.
\item If $Z^0\in\J_m|Y(h(n))$ then $Z^2\notin\J_m|Y(h(n))$, so there is $k\in\omega$ with $C^1_{(k)}\notin\J_m|Y(h(n))$. Put $j(n)=k+1$. Similarly as above, one can conclude that there is $Z'\subseteq Y(h(n))\cap C^1_{(k)}$ with $Z'\in\overline{\fin^{2m+4}}(k)\setminus\J_m|(Y(h(n))\cap C^1_{(k)})$. Actually, since $Z'\setminus\bigcup_{j\leq l}C^2_{(k,j)}\notin\J_m|Y(h(n))$, for every $l\in\omega$ (by the choice of $m$), we can find $Z_n\subseteq Y(h(n))\cap C^1_{(k)}$ with $Z_n\in\overline{\emptyset\otimes\fin^{2m+3}}(k)\setminus\J_m|(Y(h(n))\cap C^1_{(k)})$. 
\end{itemize}
Note that $Z_n\setminus \bigcup_{i\in\omega}\bigcup_{j\leq l}C^2_{(i,j)}\notin\J_m|Y(h(n))$, for every $l\in\omega$ (by the choice of $m$).

We are ready to define the required set $W\in\overline{\fin^{2m+5}}\setminus\fin'_\omega$:
$$W=\bigcup_{i\in\omega}\bigcup_{j\in\omega}\left(C^2_{(i,j)}\cap\bigcup\left\{Z_n:\ n\leq j,j(n)\in\{0,i+1\}\right\}\right).$$
It is easy to see that 
$$\bigcup_{i\in\omega}\bigcup_{j\in\omega}\left(C^2_{(i,j)}\cap\bigcup\left\{Z_n:\ n\leq j,j(n)=0\right\}\right)\in\overline{\emptyset^2\otimes\fin^{2m+2}}\subseteq\overline{\fin^{2m+5}}$$
and $\bigcup_{j\in\omega}\left(C^2_{(i,j)}\cap\bigcup\left\{Z_n:\ n\leq j,j(n)=i+1\right\}\right)$ is in $\overline{\emptyset\otimes\fin^{2m+3}}(i)$, for each $i\in\omega$. Hence, 
$$\bigcup_{i\in\omega}\bigcup_{j\in\omega}\left(C^2_{(i,j)}\cap\bigcup\left\{Z_n:\ n\leq j,j(n)=i+1\right\}\right)\in\overline{\fin^{2m+5}},$$
which implies that $W\in\overline{\fin^{2m+5}}$.

To end the proof, we have to explain why $W\notin\fin'_\omega$. Suppose to the contrary that $W\in\fin'_\omega$. By Lemma \ref{lem}, there is $k\in\omega$ with $W\in\J_k$. Since $\J_n\subseteq\J_{n+1}$, for all $n\in\omega$ (by Lemma \ref{lem}), we may assume that $k>m$. By the definition of $\J_k$, $W\in\J_k$ means that there are $M_0\subseteq\omega,M_1\subseteq\omega^2,\ldots,M_k\subseteq\omega^{k+1}$ with $M_i\in(\fin^{i+1})^\star$, for $i=0,\ldots,k$, such that if $s_i\in M_i$, for $i=0,\ldots,k$, then $W\cap X_{s_0}\cap\ldots\cap X_{s_k}\in\fin$. Pick any sequence $(s_{m+1},\ldots,s_k)\in M_{m+1}\times\ldots\times M_k$ and denote $n=h^{-1}((s_{m+1},\ldots,s_k))$. As $(Z_n\setminus\bigcup_{i\in\omega}\bigcup_{j<n}C^2_{(i,j)})\subseteq W$ and $Z_n\subseteq Y(h(n))$, we get 
$$W\cap X_{s_0}\cap\ldots\cap X_{s_k}=W\cap X_{s_0}\cap\ldots\cap X_{s_m}\cap Y(h(n))\supseteq\left(Z_n\setminus\bigcup_{i\in\omega}\bigcup_{j<n}C^2_{(i,j)}\right)\cap X_{s_0}\cap\ldots\cap X_{s_m}.$$
Since $Z_n\setminus\bigcup_{i\in\omega}\bigcup_{j<n}C^2_{(i,j)}\notin\J_m$, there is $(s_0,\ldots,s_m)\in M_0\times\ldots\times M_m$ such that $(Z_n\setminus\bigcup_{i\in\omega}\bigcup_{j<n}C^2_{(i,j)})\cap X_{s_0}\cap\ldots\cap X_{s_m}\notin\fin$. Hence, $W\cap X_{s_0}\cap\ldots\cap X_{s_k}\notin\fin$, which contradicts the choice of $M_0,\ldots,M_k$ and proves that $W\notin\fin'_\omega$. 
\end{proof}

\section{The modified conjecture}

In this section we discuss a possible modification of Conjecture \ref{hip}. However, we start with a nice feature of the ideal $\fin'_\omega$.

\begin{theorem}
\label{mc}
The following conditions are equivalent for any ideal $\I$:
\begin{itemize}
\item[(a)] $\fin'_\omega\sqsubseteq\I$;
\item[(b)] $\fin'_\omega\leq_K\I$;
\item[(c)] $\fin^{n+1}\sqsubseteq\I$ for all $n\in\omega$;
\item[(d)] $\fin^{n+1}\leq_K\I$ for all $n\in\omega$.
\end{itemize}
\end{theorem}

\begin{proof}
(a)$\implies$(b): Obvious.

(b)$\implies$(d): For $n\geq 2$ this follows directly from the definition of $\fin'_\omega$ and for $n=1$ this is satisfied for every ideal.

(d)$\implies$(c): This is \cite[Example 4.1]{Kat}.

(c)$\implies$(a): Without loss of generality we may assume that $\I$ is an ideal on $\omega$. By \cite[Lemma 3.3]{Kat} to prove $\fin'_\omega\sqsubseteq\I$ it suffices to define an injection $f:\omega\to\omega$ witnessing that $\fin'_\omega\leq_K\I$. 

For each $n\in\omega$ let $h_n:\omega^{n+2}\to\omega$ denote the bijection witnessing that $\fin^{n+2}\sqsubseteq\I$. Given $i,n\in\omega$ let $g_i(n)\in\omega^{n+1}$ be the unique index such that $i\in h_n[\{g_i(n)\}\times\omega]$ (i.e., $g_i(n)$ indicates in which element of the partition $\{h_n[\{s\}\times\omega]:s\in\omega^{n+1}\}$ number $i$ is).

We will define $f$ inductively. At first let $f(0)$ be any element of $X_{g_0(0)}\cap X_{g_0(1)}\cap\ldots\cap X_{g_0(g_0(0))}$. Assume now that $f(j)$ for all $j<i$ are already defined. Let $f(i)$ be any element of 
$$X_{g_i(0)}\cap X_{g_i(1)}\cap\ldots\cap X_{g_i(g_i(0))}\setminus\{f(j):j<i\}.$$ 
Note that this is possible as $X_{g_i(0)}\cap X_{g_i(1)}\ldots\cap X_{g_i(g_i(0))}$ is infinite.

The important observation is that for each $n\in\omega$ and $i\in\omega\setminus\bigcup_{j<n}h_0[\{j\}\times\omega]$ we have $g_i(0)\geq n$, so $f(i)\in X_{g_i(0)}\cap\ldots\cap X_{g_i(n)}\subseteq X_{g_i(n)}$. Now we will check that $f$ is as needed. Since $\I$ is an ideal, it suffices to check that $f^{-1}[A]\in\I$ for all $A\in\fin^{n+2}(\{X_s:s\in\omega^{n+1}\})$ and $n\in\omega$. 

Fix $n\in\omega$ and $A\in\fin^{n+2}(\{X_s:s\in\omega^{n+1}\})$. Then $A=A'\cup A''$, where $A'\cap X_s\in\fin$ for all $s\in\omega^{n+1}$ and $A''\subseteq \bigcup_{s\in B} X_s$ for some $B\in\fin^{n+1}$. Observe that 
$$f^{-1}[X_s]\subseteq h_n[\{s\}\times\omega]\cup\bigcup_{j<n}h_0[\{j\}\times\omega]$$
 for all $s\in\omega^{n+1}$: for each $i\notin\bigcup_{j<n}h_0[\{j\}\times\omega]$ we have $f(i)\in X_{g_i(0)}\cap\ldots\cap X_{g_i(n)}$, so $f(i)\in X_s$ implies $s=g_i(n)$ and $i\in h_n[\{s\}\times\omega]$, since $\{h_n[\{t\}\times\omega]:t\in\omega^{n+1}\}$ is a partition of $\omega$. Thus, 
$$f^{-1}[A'']\subseteq h_n[B\times\omega]\cup\bigcup_{j<n}h_0[\{j\}\times\omega].$$
 
Moreover, 
$$C=f^{-1}[A']\setminus \bigcup_{j<n}h_0[\{j\}\times\omega]\in\I$$ 
as $C\cap h_n[\{s\}\times\omega]\in\fin$ for all $s\in\omega^{n+1}$. Indeed, if $i\in C\cap h_n[\{s\}\times\omega]$ then $s=g_i(n)$ (by the definition of $g_i(n)$) and using the above observation we have $f(i)\in A'\cap X_{s}$. Since $A'\cap X_{s}\in\fin$ and $f$ is injective, we conclude that $C\cap h_n[\{s\}\times\omega]\in\fin$.

Therefore,
$$f^{-1}[A]=f^{-1}[A']\cup f^{-1}[A'']\subseteq \bigcup_{j<n}h_0[\{j\}\times\omega]\cup C\cup h_n[B\times\omega]\in\I$$
(since $B\times\omega\in\fin^{n+2}$). 
\end{proof}

\begin{remark}
In the language of \cite{Kat}, equivalence of items (a) and (b) in Theorem \ref{mc} means that $\fin'_\omega$ has property Kat.
\end{remark} 

\begin{corollary}
$\fin'_\omega\sqsubseteq\fin_\omega$.
\end{corollary}

\begin{proof}
It is easy to check that $\fin^{n+1}\sqsubseteq\fin_\omega$, for all $n\in\omega$. Thus, the thesis follows from Theorem \ref{mc}.
\end{proof}

In the light of Theorem \ref{mc}, there is a chance of reducing Conjecture \ref{hip} to the case of successor ordinals. We believe that the following can be proved using similar methods to those from this paper.

\begin{conjecture}
\label{hip2}
For each limit ordinal $0<\alpha<\omega_1$ there is a $\bf{\Sigma^0_{\alpha}}$ ideal $\fin'_\alpha$ of rank $\alpha$ such that the following conditions are equivalent for any ideal $\I$:
\begin{itemize}
\item[(a)] $\fin'_\alpha\sqsubseteq\I$;
\item[(b)] $\fin'_\alpha\leq_K\I$;
\item[(c)] $\fin_{\beta}\sqsubseteq\I$ for all successor ordinals $\beta<\alpha$;
\item[(d)] $\fin_{\beta}\leq_K\I$ for all successor ordinals $\beta<\alpha$.
\end{itemize}
\end{conjecture} 

By Theorems \ref{x} and \ref{mc}, the above is true for $\alpha=\omega$. 

Conjecture \ref{hip2} would reduce Conjecture \ref{hip} to the case of successor ordinals:

\begin{remark}
Assume that Conjecture \ref{hip2} is true. Then Conjecture \ref{hip} is true for all successor ordinals if and only if for each analytic ideal $\I$:
\begin{itemize}
\item $\text{rk}(\mathcal{I})\geq\alpha$ is equivalent to $\fin_\alpha\sqsubseteq\I$, for all successor ordinals $0<\alpha<\omega_1$;
\item $\text{rk}(\mathcal{I})\geq\alpha$ is equivalent to $\fin'_\alpha\sqsubseteq\I$, for all limit ordinals $0<\alpha<\omega_1$.
\end{itemize}
\end{remark}


\end{document}